\documentclass[a4 paper,12pt]{article}
\usepackage{latexsym}
\usepackage{bm}
\usepackage{amsmath,amssymb,amsfonts,amsthm,graphicx}
\usepackage{times}
\usepackage{color}
\usepackage{textcomp}
\usepackage{pifont}
\usepackage{bbding}
\theoremstyle{plain}
\newtheorem{theorem}{Theorem}[section]
\newtheorem{lemma}[theorem]{Lemma}
\newtheorem{corollary}[theorem]{Corollary}
\newtheorem{proposition}[theorem]{Proposition}
\theoremstyle{definition}
\newtheorem{definition}[theorem]{Definition}
\newtheorem{remark}[theorem]{Remark}
\newtheorem{myexample}[theorem]{Example}
\title{Eulerian-minors and a concise recursive characterization of 4-regular planar graphs}
\author{
 Metrose Metsidik and Qi Yan\footnote{Corresponding author. Email: qiyanmath@163.com}\\
{\small College of Mathematical Sciences, Xinjiang Normal University, $830017$, PR China}\\
{\small School of Mathematics and Statistics, Lanzhou University, Lanzhou, Gansu $730000$, PR China}}
\date{}

\begin{document}

\maketitle
\baselineskip=0.30in

\begin{abstract}
An Eulerian-minor of an Eulerian graph is obtained from an Eulerian subgraph of the Eulerian graph by contraction. The Eulerian-minor operation preserves Eulerian properties of graphs and moreover Eulerian graphs are well-quasi-ordered under Eulerian-minor relation. In this paper, we characterize Eulerian, planar and outer-planar Eulerian graphs by means of excluded Eulerian-minors, and provide a concise recursive characterization to 4-regular planar graphs.
\vskip 0.2cm

\end{abstract}

\section{Introduction}
\noindent

A {\it minor} of a graph is obtained from a subgraph of the graph by contraction. Robertson and Seymour~\cite{R-S} showed that graphs are well-quasi-ordered under the graph minor relation. This is a profoundest result in graph theory and be called graph minor theorem. By graph minor theorem, we know that every minor-closed families of graphs can be characterised via excluded minors, but very few explicit characterisations are known. Perhaps the best-known is Wagner theorem~\cite{Wagner}: a graph $G$ is planar if and only if it contains neither $K_5$ nor $K_{3,3}$ as a minor.

It is easy to check that a minor of an Eulerian or a bipartite graph is not always Eulerian or bipartite. Chudnovsky et al.~\cite{Chudnovsky} introduced a bipartite-minor  operation to bipartite graphs and proved a bipartite analog of Wagner theorem: a bipartite graph is planar if and only if it does not contain $K_{3,3}$ as a bipartite-minor, meanwhile they left a bipartite analog of graph minor theorem as an open problem: is the bipartite-minor relation a well-quasi-ordering on the set of bipartite graphs? Wagner~\cite{Wagner1} defined bipartite and Eulerian minors of binary matroids and extended Chudnovsky et al's bipartite planar graph characterization to binary matroids.

Moffatt~\cite{Moffatt6} presented ribbon graph (cellularly embedded graph) minors. Ribbon graph minors differ from abstract graph minors only in contracting loops. Contracting loops are necessary in ribbon graph minors and taking this may produce extra vertices and components. This derives a conclusion that the underlying graphs of ribbon graph minors need not be abstract graph minors. Moffatt conjectured that ribbon graphs are well-quasi-ordered under the ribbon graph minor relation, then he characterized the ribbon graphs representing link diagrams~\cite{Moffatt6} and the ribbon graphs admitting partial duals of Euler genus at most one~\cite{Moffatt7}. Metsidik and Jin~\cite{MM}  defined Eulerian and even-face ribbon graph minors and described Eulerian, even-face, plane Eulerian and plane even-face ribbon graphs in terms of excluded Eulerian/even-face ribbon graph minors.

In 2018, Wagner introduce that an Eulerian minor is gained by either edge contractions, admissible splits, or isolated vertex deletions. One can also define Eulerian graph minors in exactly the same of graph minors. An {\it Eulerian-minor} of an Eulerian graph is obtained from an Eulerian subgraph of the Eulerian graph by a sequence of edge contractions. This Eulerian-minor is different from Wagner's Eulerian minor, for example, Eulerian subgraphs of an Eulerian graph are obviously Eulerian-minors of the Eulerian graph, but this doesn't hold true for Wagner's Eulerian minor. In the following two sections of this paper, we show the following results:
\begin{enumerate}
  \item A graph is Eulerian if and only if it contains no Eulerian-minor equivalent to $K_2$;
  \item Eulerian graphs are well-quasi-ordered under the Eulerian-minor relation;
  \item An Eulerian graph is planar if and only if it contains no Eulerian-minor* equivalent to $K_5$ or $K_{3,3}^\prime$;
  \item An Eulerian graph is outerplanar if and only if it contains no Eulerian-minor* equivalent to $K_{2,3}^\prime$ or $K_{4}^\prime$.
\end{enumerate}

4-regular planar graphs are interesting not only because of their close connection to knot theory but also their intrinsic attributes. Manca proposed four planar operations to generate all 4-regular planar graphs from the octahedron graph~\cite{Manca}. Lehel corrected Manca's result by adding a fifth planar operation~\cite{Lehel}. Then Broersma et al. showed all 3-connected 4-regular planar graphs can be generated from the octahedron graph by using the three planar operations among Manca's four planar operations~\cite{Broersma}. In the last section of this paper, we generate all 4-regular planar graphs from the 4-regular graphs on one vertex by only one planar operation.

\section{Eulerian-minor}

All graphs considered in this paper are finite. An Eulerian graph is connected in general, we somewhat abuse the concept in this note. Here by an {\it Eulerian} graph we mean a graph with no odd degree vertex. Note that an empty graph (edgeless graph) is also Eulerian. A {\it cycle decomposition} of a graph $G$ is a partition of the edge-set $E(G)$ of $G$ such that each partition set forms a cycle. The following theorem is a classical characterization of Eulerian graphs.

\begin{theorem}[\cite{O.Ve}]\label{TH-1}
A graph $G$ is Eulerain if and only if $G$ has a cycle decomposition.
\end{theorem}
\begin{definition}
Let $C$ be a cycle of a graph $G$. Then the operation $G-E(C)$ is called {\it cycle-deletion}, where $E(C)$ is the edge set of $C$.
\end{definition}

Using the symbol $G/xy$ we denote the graph obtained from a graph $G$ by {\it contracting} an edge $xy$ of $G$ into a new vertex $z$ such that $z$ is adjacent to all former neighbors of $x$ and $y$.
\begin{definition}\label{MD-00}
A graph $H$ is an {\it Eulerian-minor} of a graph $G$ if there exists a sequence of graphs taking $G$ to $H$ such that each graph in the sequence is obtained from its predecessor by either an edge contraction, a cycle-deletion, or the
deletion of an isolated vertex.
\end{definition}

Clearly, any subgraph of a graph is also a minor of the graph. Since an arbitrary deletion is not allowed in Eulerian-minor, a subgraph of a graph may not be an Eulerian-minor of the graph. Anyway, we have the following proposition.

\begin{proposition}\label{MP-33}
An Eulerian subgraph of an Eulerian graph is an Eulerian-minor of the Eulerian graph.
\end{proposition}
\begin{proof}
Let $H$ be an Eulerian subgraph of an Eulerian graph $G$. Then the subgraph $G-E(H)$ is again Eulerian. By Theorem~\ref{TH-1}, $G-E(H)$ has a cycle decomposition. Set $E(G)-E(H)=\bigcup_{i=1}^sE(C_i)$, where $C_i$ is a cycle and $E(C_i)\cap E(C_j)=\emptyset$ for $i\neq j\in\{1,\cdots,s\}$. By cycle-deletions of $C_1,\cdots,C_s$ and deletions of isolated vertices not belonging to $H$, we obtain $H$ as an Eulerian-minor of $G$.
\end{proof}

\begin{remark}
By the proof of Proposition~\ref{MP-33}, we have that an Eulerian subgraph of an Eulerian graph is gained by a sequence of cycle-deletions and isolated vertex deletions. Therefore, the two definitions of Eulerian-minor described in Introduction and Definition~\ref{MD-00} are equivalent.
\end{remark}

A {\it quasi-ordering} is a reflexive, anti-symmetric and transitive relation. An {\it antichain} is a subset with no comparable elements. A quasi-ordering on a set is a {\it well-quasi-ordering} if it contains neither an infinite antichain nor an infinite decreasing sequence. The following two propositions are clear from the definition of Eulerian-minor.

\begin{proposition}\label{MP-1}
Every Eulerian-minor of an Eulerian graph is also Eulerian.
\end{proposition}

\begin{proposition}\label{MP-2}
The Eulerian-minor relation is a quasi-ordering.
\end{proposition}

In the following theorem, we characterize Eulerian graphs by means of excluded Eulerian-minor.

\begin{theorem}\label{MT-1}
A graph is Eulerian if and only if it contains no Eulerian-minor equivalent to $K_2$.
\end{theorem}
\begin{proof}
By Proposition~\ref{MP-1}, an Eulerian graph cannot have an Eulerian-minor equivalent to $K_2$. In the following, we prove the sufficiency by showing inverse negative proposition.

Suppose that $G$ is not Eulerian. By Theorem~\ref{TH-1}, the edge set of $G$ cannot be partitioned into cycles. Therefore, by cycle-deletions of all edge-disjoint cycles of $G$, we obtain a nonempty forest $F$ as an Eulerian-minor of $G$. Clearly, $F$ has an Eulerian-minor equivalent to $K_2$. By Proposition~\ref{MP-2}, $G$ has an Eulerian-minor equivalent to $K_2$.
\end{proof}

From the above two proofs we can see that cycle-deletions and isolated vertex deletions are enough deletion operations for the family of Eulerian graphs. By graph minor theorem we have the following analogical result.

\begin{theorem}
Eulerian graphs are well-quasi-ordered under the Eulerian-minor relation.
\end{theorem}

\section{Characterizing planar Eulerian graphs}

A cycle $C$ of a graph $G$ is {\it non-separating} if the removal of the vertices of $C$ from $G$ does not increase the number of connected components, and $C$ is {\it induced} if each two nonadjacent vertices of $C$ are not connected by an edge in $G$. Induced non-separating cycles are known in the literature as {\it peripheral}
cycles.

It is sometimes convenient to consider an edge consisted of two half edges. In this section, just for the simplicity of results we add an operation to Eulerian-minor.

\begin{definition}\label{MD11}
Let $e_1$ and $e_2$ be any two half-edges incident to a vertex $v$ of a graph $G$. Then the operation that of deleting the two half-edges and gluing the remaining two half edges together to form a new edge is called {\it demotion} of the vertex $v$. The demotion of $v$ is {\it admissible} if $e_1ve_2$ is a part of a peripheral cycle of $G$.
\end{definition}

In Figure~1, we pictorially illustrate an admissible demotion of a vertex $v$, where $v_1vv_2v_1$ is the peripheral cycle containing $e_1ve_2$. In this case, the demotion of $v$ equals $\big(G-\{v_1v,v_2v\}\big)\cup\{v_1v_2\}$. Notice that a demotion equals deleting a loop if the deleted two half-edges consist the loop.
\begin{figure}[htbp]
\centering
\includegraphics[width=6.0cm]{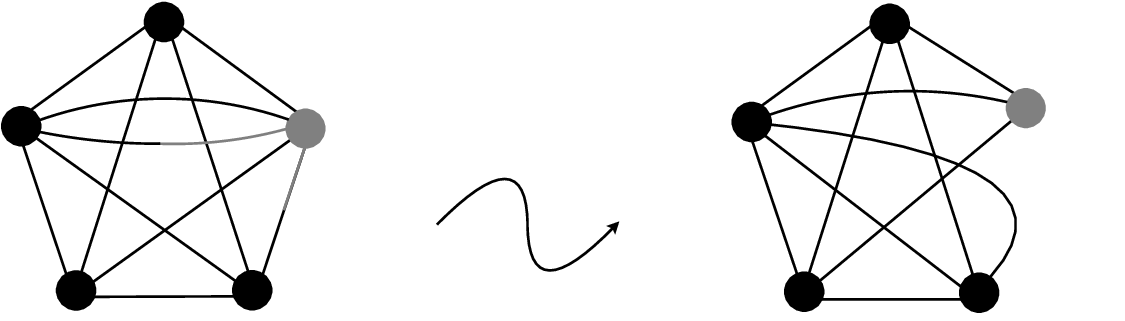}
\put(-180,30){$v_1$}
\put(-129,0){$v_2$}
\put(-121,28){$v$}
\put(-137,27){{\small $e_1$}}
\put(-128,20){{\small $e_2$}}
\put(-69,30){$v_1$}
\put(-18,0){$v_2$}
\put(-11,30){$v$}
\put(-145,-15){$G$}
\put(-70,-15){$\big(G-\{v_1v,v_2v\}\big)\cup\{v_1v_2\}$}
\caption{An admissible demotion operation.}
\end{figure}
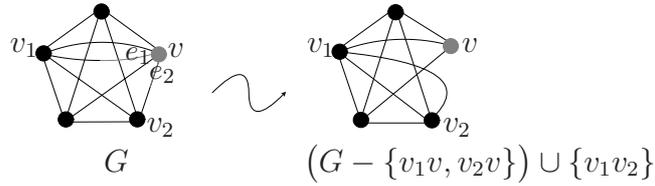
\begin{definition}
A graph $H$ is an {\it Eulerian-minor}* of
a graph $G$ if $H$ is obtained from $G$ by a sequence of contractions, cycle-deletions, isolated vertex deletions and admissible demotions.
\end{definition}

The following result is well-known.
\begin{lemma}[\cite{{Bondy}}]\label{PO-1}
Minors of planar graphs are planar.
\end{lemma}
We have an analogist result.
\begin{proposition}\label{MP-3}
Eulerian-minors* of planar graphs are planar.
\end{proposition}
\begin{proof}
By Lemma~\ref{PO-1}, we only need to check that an admissible demotion preserves the planarity. If the admissible demotion in considering equals the deletion of a loop, then we are done. Hence we assume that the two half edges $e_1$ and $e_2$ defined in Definition~\ref{MD11} come from two different edges $v_1v$ and $vv_2$. Consider an embedding of $G$ in the plane $\mathbb{R}^2$, and let $C$ be a peripheral cycle containing $v_1vv_2$. By Jordan-Sch\"{o}nflies Theorem~\cite{Bondy}, $C$ partitions the rest of $\mathbb{R}^2$ into two
disjoint arcwise-connected open sets. As $C$ is peripheral, one of these sets contains no vertices and no edges of $G$, and hence it is a face of this embedding of $G$. Thus, in this embedding of $G$, $v_1v$ and $vv_2$ are successive in the cyclic ordering of the edges at $v$, and therefore the admissible demotion of $v$ produces an embedding of the resulting graph $\big(G-\{v_1v,v_2v\}\big)\cup\{v_1v_2\}$ in $\mathbb{R}^2$.
\end{proof}

\begin{figure}[htbp]
\centering
\includegraphics[width=3.0cm]{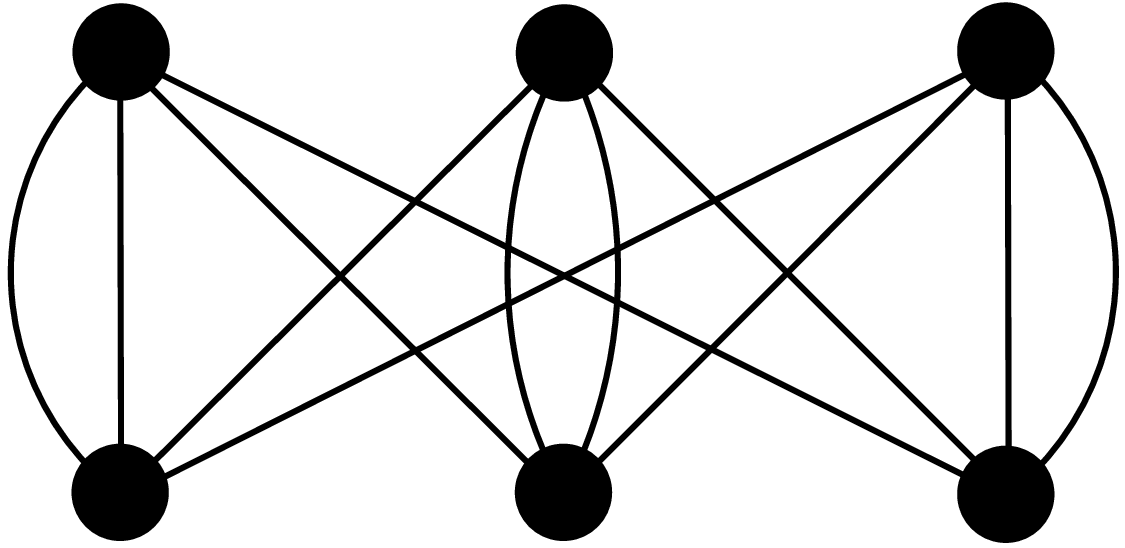}
\caption{$K_{3,3}^\prime$.}
\end{figure}

\begin{theorem}\label{MT-2}
Let $K_{3,3}^\prime$ be the graph shown in Figure~2. Then an Eulerian graph is planar if and only if it contains no Eulerian-minor* equivalent to $K_5$ or $K_{3,3}^\prime$.
\end{theorem}
\begin{proof}
It is well-known that $K_5$ and $K_{3,3}^\prime$ are non-planar graphs. By Proposition~\ref{MP-3}, a planar Eulerian graph cannot contain $K_5$  or $K_{3,3}^\prime$ as an Eulerian-minor*.

For the converse, we suppose that the Eulerian graph $G$ in considering is not planar. Then, by Kuratowski Theorem~\cite{Thomassen}, $G$ contains a subgraph $H$ isomorphic to a subdivision of either $K_5$ or $K_{3,3}$. By contracting the subdivided edges of $H$, we obtain an Eulerian graph $G_1$ as an Eulerian-minor of $G$ such that $G_1$ has a subgraph $H_1$ isomorphic to either $K_5$ or $K_{3,3}$.

{\it Case A}. $H_1\cong K_5$.

Recall that $K_5$ is an Eulerian subgraph of the Eulerian graph $G_1$ and $G_1$ is an Eulerian-minor of $G$. Then, by Propositions~\ref{MP-33}~and~\ref{MP-2}, $G$ has an Eulerian-minor equivalent to $K_5$.

{\it Case B}. $H_1\cong K_{3,3}$.

We find a spanning forest $F$ of $G_1-E(H_1)$ rooted with the vertices of $H_1$ for the component of $G_1$ containing $H_1$. Notice that $F$ has six trees and each of these trees exactly contains one vertex of $H_1$. We can easily find such a rooted spanning forest just by deleting some edges on a spanning tree of $G_1$. By contracting all edges of $F$, we obtain an Eulerian-minor $G_2$ of $G$ such that $G_2$ has a component $G_3$ on $|V(H_1)|$ vertices and including $H_1$ as a subgraph. By cycle-deletions and isolated vertex deletions, we force $G_3$ to be an Eulerian-minor of $G$. By Proposition~\ref{MP-1}, $G_3$ is Eulerian. Again by cycle-deletions of $G_3$, we obtain an Eulerian-minor $G_4$ of $G$ having minimal number of edges with respect to containing $H_1$. Notice that $G_4$ is also an Eulerian graph.

Let $\{a_1,a_2,a_3\}$ and $\{b_1,b_2,b_3\}$ be the bipartition of $H_1$. If $a_ia_j,b_sb_t\in E(G_4)$ for $i\neq j,s\neq t$, say $a_1a_2,b_1b_2\in E(G_4)$, then $G_4/ a_3b_3$ has a subgraph isomorphic to $K_5$. Then, by Propositions~\ref{MP-33}~and~\ref{MP-2}, $G$ has an Eulerian-minor equivalent to $K_5$. Hence, in the following, we assume that $b_1b_2, b_1b_3,b_2b_3\notin E(G_4)$.

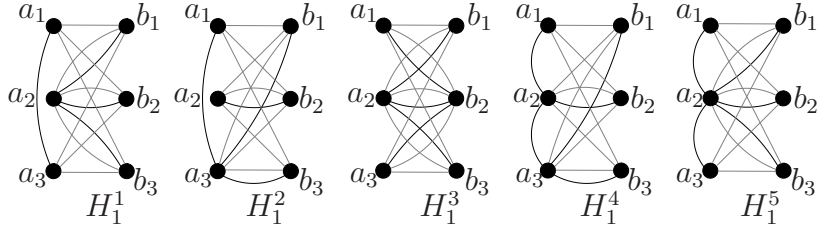
\begin{figure}[htbp]
\centering
\includegraphics[width=10.0cm]{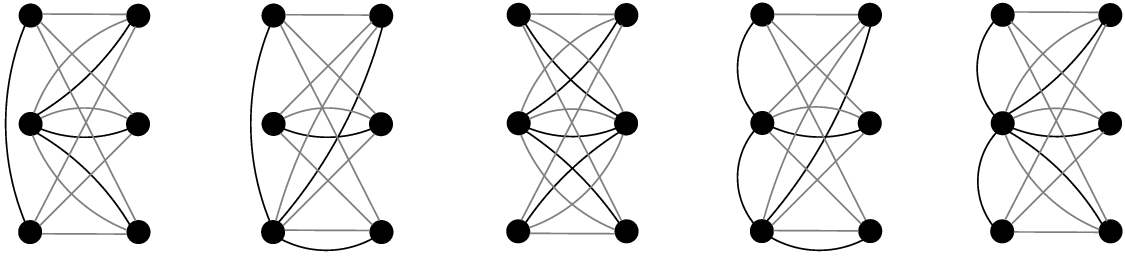}
\put(-290,65){$a_1$}
\put(-246,62){$b_1$}
\put(-228,65){$a_1$}
\put(-184,62){$b_1$}
\put(-166,65){$a_1$}
\put(-122,62){$b_1$}
\put(-104,65){$a_1$}
\put(-61,62){$b_1$}
\put(-44,65){$a_1$}
\put(0,62){$b_1$}
\put(-294,33){$a_2$}
\put(-246,31){$b_2$}
\put(-232,33){$a_2$}
\put(-185,31){$b_2$}
\put(-167,33){$a_2$}
\put(-123,31){$b_2$}
\put(-105,33){$a_2$}
\put(-61,31){$b_2$}
\put(-44,33){$a_2$}
\put(0,31){$b_2$}
\put(-290,3){$a_3$}
\put(-247,0){$b_3$}
\put(-228,3){$a_3$}
\put(-185,0){$b_3$}
\put(-166,3){$a_3$}
\put(-123,0){$b_3$}
\put(-104,3){$a_3$}
\put(-62,0){$b_3$}
\put(-44,3){$a_3$}
\put(-1,0){$b_3$}
\put(-265,-10){$H_1^1$}
\put(-205,-10){$H_1^2$}
\put(-140,-10){$H_1^3$}
\put(-80,-10){$H_1^4$}
\put(-20,-10){$H_1^5$}
\caption{All possible configurations of $G_4~(\ncong K_{3,3}^\prime)$ in the case $b_1b_2, b_1b_3,b_2b_3\notin E(G_4)$.}
\end{figure}

In Figures~2 and~3, we illustrate all possible configurations of $G_4$ according to the degree sequence of $G_4$. Since $H_1$ is a subgraph of $G_4$, we only need to consider all possible configurations of $G_4-E(H_1)$. If $G_4$ is 4-regular (the degree sequence of $G_4$ is $(4,4,4,4,4,4)$), then the degree sequence of $G_4-E(H_1)$ is $(1,1,1,1,1,1)$ and therefor there is only one configuration of $G_4-E(H_1)$, i.e., $G_4\cong K_{3,3}^\prime$; There are exactly two configurations for the degree sequence $(4,4,4,4,4,6)$ of $G_4$ (resp. $(1,1,1,1,1,3)$ of $G_4-E(H_1)$), see Graphs $H_1^1$ and $H_1^2$ in Figure~3, where the edges of $H_1$ are grayed; There are also exactly two configurations if the degree sequence of $G_4$ (resp. $G_4-E(H_1)$) is $(4,4,4,4,6,6)$ (resp. $(1,1,1,1,3,3)$) as in Graphs $H_1^3$ and $H_1^4$; There is again only one configuration if the degree sequence of $G_4$ (resp. $G_4-E(H_1)$) is $(4,4,4,4,4,8)$ (resp. $(1,1,1,1,1,5)$), see Graph $H_1^5$ in Figure~3; If $G_4$ has more degree 6 or degree 8 or even more greater degree vertices, then $|E(G_4)\setminus E(H_1)|\geq 6$, and therefore $G_4-E(H_1)$ contains cycles. Taking cycle-deletions for these cycles we obtain an Eulerian-minor still containing $H_1$ as a subgraph. This contradicts the minimality of $G_4$.

For each of the five graphs in Figure~3, we describe the sequences of contractions and admissible demotions yielding $K_5$ as an Eulerian-minor*. For each demotion we indicate a peripheral cycle showing that the demotion is admissible, and call such a peripheral cycle a {\it witness cycle}.

{\it Case B.$1$}. $H_1^1$.

\begin{enumerate}
  \item $H_1^{11}:=\big(H_1^1-\{b_1a_2,a_2b_2\}\big)\cup\{b_1b_2\}$ and $b_1a_2b_2a_1b_1$ is a witness cycle;
  \item $H_1^{11}/a_2b_3$.
\end{enumerate}

{\it Case B.$2$}. $H_1^2$.

\begin{enumerate}
  \item $H_1^{21}:=\big(H_1^2-\{b_1a_3,a_3b_3\}\big)\cup\{b_1b_3\}$ and $b_1a_3b_3a_1b_1$ is a witness cycle;
  \item $H_1^{21}/a_2b_2$.
\end{enumerate}

{\it Case B.$3$}. $H_1^3$.

\begin{enumerate}
  \item $H_1^{31}:=\big(H_1^3-\{b_1a_2,a_2b_2\}\big)\cup\{b_1 b_2\}$ and $b_1a_2b_2a_1b_1$ is a witness cycle;
   \item $H_1^{32}:=\big(H_1^{31}-\{a_1b_2,b_2a_3\}\big)\cup\{a_1 a_3\}$ and $a_1b_2a_3b_1a_1$ is a witness cycle;
  \item $H_1^{32}/a_2b_3$.
\end{enumerate}

{\it Case B.$4$}. $H_1^4$.

\begin{enumerate}
  \item $H_1^{41}:=\big(H_1^4-\{b_1a_3,a_3b_3\}\big)\cup\{b_1 b_3\}$ and $b_1a_3b_3a_1b_1$ is a witness cycle;
  \item $H_1^{41}/a_1b_2$.
\end{enumerate}

{\it Case B.$5$}. $H_1^5$.

\begin{enumerate}
  \item $H_1^{51}:=\big(H_1^5-\{b_1a_2,a_2b_2\}\big)\cup\{b_1 b_2\}$ and $b_1a_2b_2a_1b_1$ is a witness cycle;
  \item $H_1^{51}/a_1b_3$.
\end{enumerate}

\end{proof}

By Jordan-Sch\"{o}nflies Theorem and induction, we can easily construct planar Eulerian graphs from planar simple graphs by adding edges. Notice that there is a surjection from planar Eulerian graphs onto planar simple graphs. Hence Theorem~\ref{MT-2} can be considered as an Eulerian-minor analog of Wagner theorem.

A graph is {\it outer-planar} if it has an embedding in the plane such that all of the vertices lie on the outer boundary. A graph is outer-planar if and only if it contains neither $K_4$ nor $K_{2,3}$ as a minor \cite{Chartrand}. Here is an Eulerian-minors* version of this result for outer-planar Eulerian graphs. And one can show it by similar methods in the proofs of Proposition~\ref{MP-3} and Theorem~\ref{MT-2}.
\begin{proposition}
Eulerian-minors* of outer-planar graphs are outer-planar.
\end{proposition}
\begin{figure}[htbp]
\centering
\includegraphics[width=6.0cm]{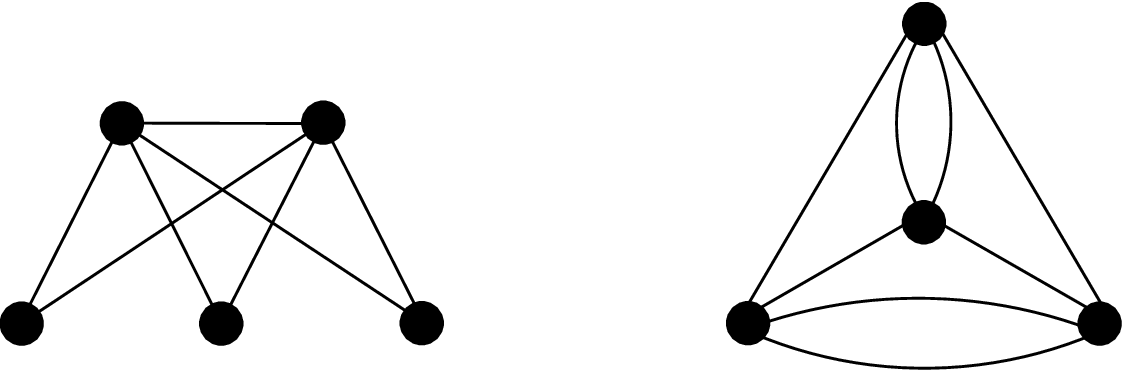}
\put(-145,-15){$K_{2,3}^\prime$}
\put(-35,-15){$K_{4}^\prime$}
\caption{Forbidden Eulerian-minors of outer-planar graphs.}
\end{figure}
\begin{theorem}
Let $K_{2,3}^\prime$ and $K_{4}^\prime$ be the two graphs shown in Figure~4. Then an Eulerian graph is outer-planar if and only if it contains no Eulerian-minor* equivalent to $K_{2,3}^\prime$ or $K_{4}^\prime$.
\end{theorem}
\section{4-regular planar graphs}
A {\it bouquet} $B_n$ is a graph having exactly one vertex and $n$ loops, and a {\it generalized bouquet} $\mathcal{B}_n$ is obtained from $B_n$ by replacing the loops with cycles.
\begin{proposition}\label{MP-4}
Let $\mathcal{B}_n$ be a generalized bouquet with $n\geq3$. Then a non-empty Eulerian graph not containing any $\mathcal{B}_n$ has a peripheral cycle.
\end{proposition}
\begin{proof}
We prove the result by induction on the number of edges $|E(G)|$. If $|E(G)|=1$, then $G$ has a loop. Obviously, the loop is peripheral. Now we assume that the result holds for all such Eulerian graphs having fewer edges than $G$. In the following we seek a peripheral cycle for $G$. By Theorem~\ref{TH-1}, $G$ has cycles, and therefore $G$ has induced cycles. Let $C$ be an induced cycle of $G$. Clearly, $G-C$ is Eulerian and does not contain a generalized bouquet $\mathcal{B}_n$ as a subgraph. Then, by induction hypothesis, $G-E(C)$ has a peripheral cycle $C_1$. There is nothing to prove for peripheral cycle $C$ or $C_1$ in $G$. Hence we now assume that both $C$ and $C_1$ are non-peripheral in $G$. And then $V(C)\cap V(C_1)\neq\emptyset$, otherwise $C_1$ is peripheral in $G$.

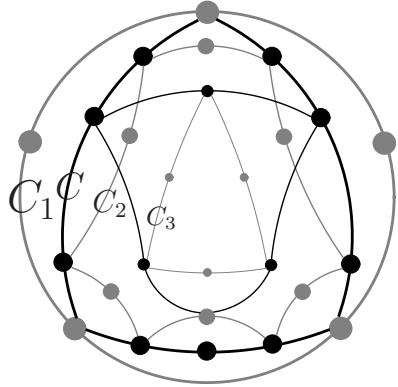
\begin{figure}[htbp]
\centering
\includegraphics[width=5.0cm]{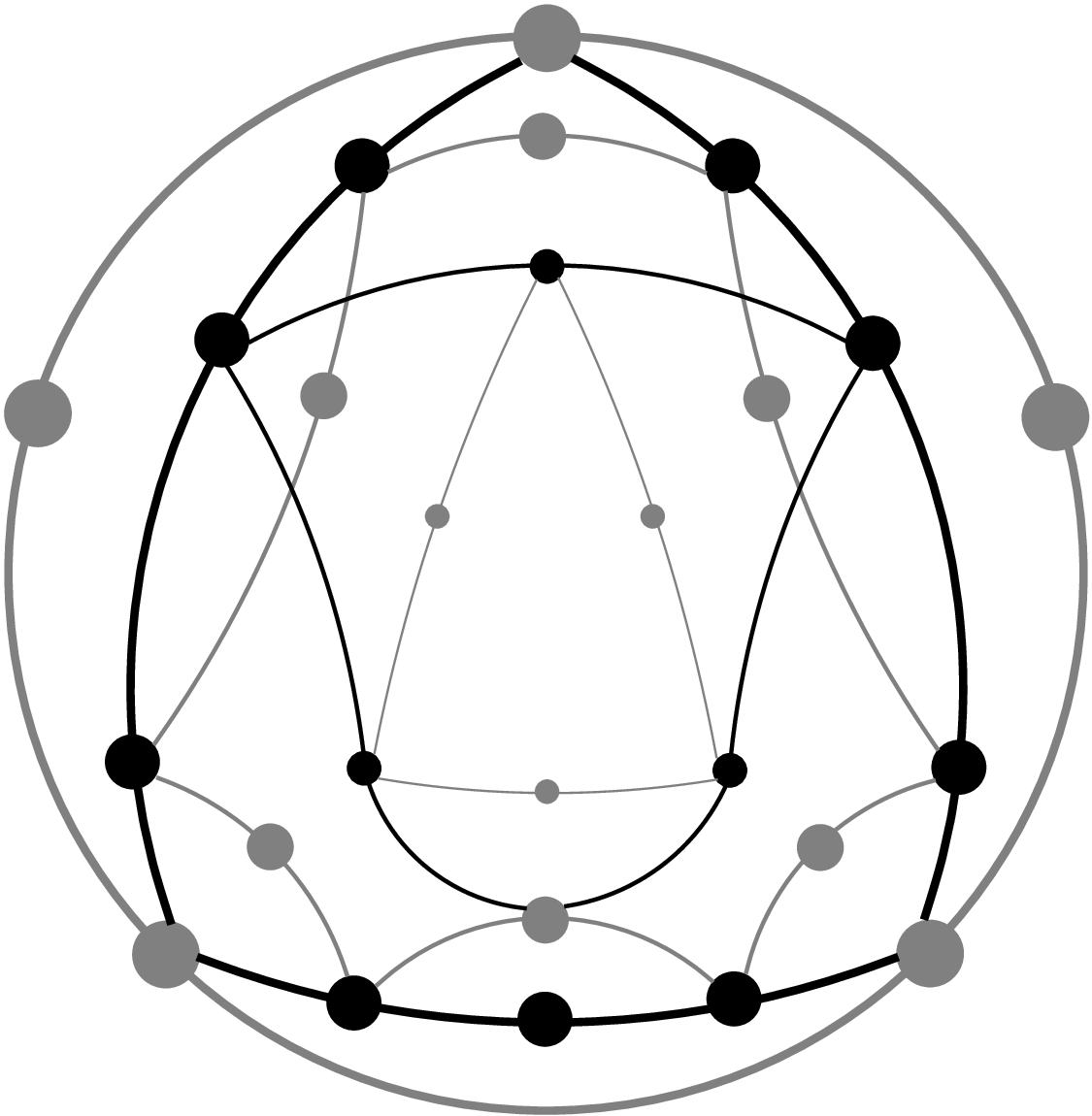}
\put(-146,65){{\Large $C_1$}}
\put(-128,70){{\large$C$}}
\put(-114,65){{$C_2$}}
\put(-94,60){{\footnotesize $C_3$}}
\caption{A local structure for seeking a peripheral cycle.}
\end{figure}

Case A. $G-V(C_1)$ has a newly created component $K$.

Since $C_1$ is a peripheral cycle of $G-E(C)$, $G-E(C)$ has a component $K^\prime$ such that $K^\prime\subseteq K$. Since $G$ is Eulerian and contains no generalized bouquet $\mathcal{B}_n$, $K^\prime$ is also Eulerian and does not contain any generalized bouquet $\mathcal{B}_n$. Again by induction hypothesis, $K^\prime$ has a peripheral cycle $C_2$. If $C_2$ is a separating cycle of $G$, then $K^\prime-E(C_2)$ has an Eulerian component $K^{\prime\prime}$  such that $K^{\prime\prime}\subseteq K^\prime$, and then we find a peripheral cycle $C_3$ for $K^{\prime\prime}$, see Figure~5 for an example. Since $G$ is finite, this process cannot be repeated infinitely, and therefore we obtain a non-separating cycle $C^\prime$ of $G$. If $C^\prime$ is not induced in $G$, then the union of the part of $C^\prime$ only including both two ends of a chord and the chord forms a peripheral cycle of $G$.

Case B. $C_1$ is non-separating and not induced in $G$.

We easily obtain a peripheral cycle of $G$ from $C_1$ by using a same method that of finding an induced cycle from $C^\prime$.
\end{proof}

By Proposition~\ref{MP-4}, we can take an admissible demotion operation for any non-empty Eulerian graph without any generalized bouquet $\mathcal{B}_n$ with $n\geq3$. Since the maximum degree $\Delta(\mathcal{B}_n)$ of $\mathcal{B}_n$ is $2n$, 4-regular graphs cannot contain any generalized bouquet $\mathcal{B}_n$ with $n\geq3$. Recall that contractions and admissible demotions preserve planarity. Thus we have the following concise recursive characterization of 4-regular planar graphs.

\begin{theorem}\label{MT-5}
Let $G$ be a 4-regular planar graph on $n$ vertices.
Then there is a 4-regular planar graph on $n-1$ vertices obtained from $G$ by an admissible demotion and contraction of an edge incident to the newly created degree 2 vertex.
\end{theorem}

By Theorem~\ref{MT-5}, any 4-regular planar graph can be reduced to the 4-regular planar graph on one vertex (the bouquet $B_2$). In other words, we construct any connected 4-regular planar graph from the bouquet $B_2$ and some free-loops (loops without vertices) by the inverse operations of admissible demotion and contraction.

\begin{corollary}\label{MC-5}
Let $B_2$ be the bouquet consisted of one vertex and two loops. Then any connected 4-regular planar graph can be obtained from $B_2$ and some free-loops by a series of the following operation on the plane:
\begin{enumerate}
  \item Subdividing an edge $e$ on a peripheral cycle $C$;
  \item Subdividing an edge $e^\prime\notin E(C)$ lying on the boundary of the other face/region including $e$;
  \item Merging the two degree 2 vertices into a degree 4 vertex.
\end{enumerate}
\end{corollary}

\begin{myexample}\label{ME-6}
Constructing the Octahedron from the bouquet $B_2$ is illustrated from left to right in Figure~6. We read Figure~6 in reverse order for reducing the Octahedron to $B_2$. In Figure~6, for giving some convenience to a reader, the edges of witness cycles, demoted vertices and new vertices created by subdivisions are colored with gray, and intermediate processes are also depicted in small pictures over (resp. under) arrow directing from left to right (resp. right to left).
\end{myexample}
\begin{figure}[htbp]
\centering
\includegraphics[width=12.0cm]{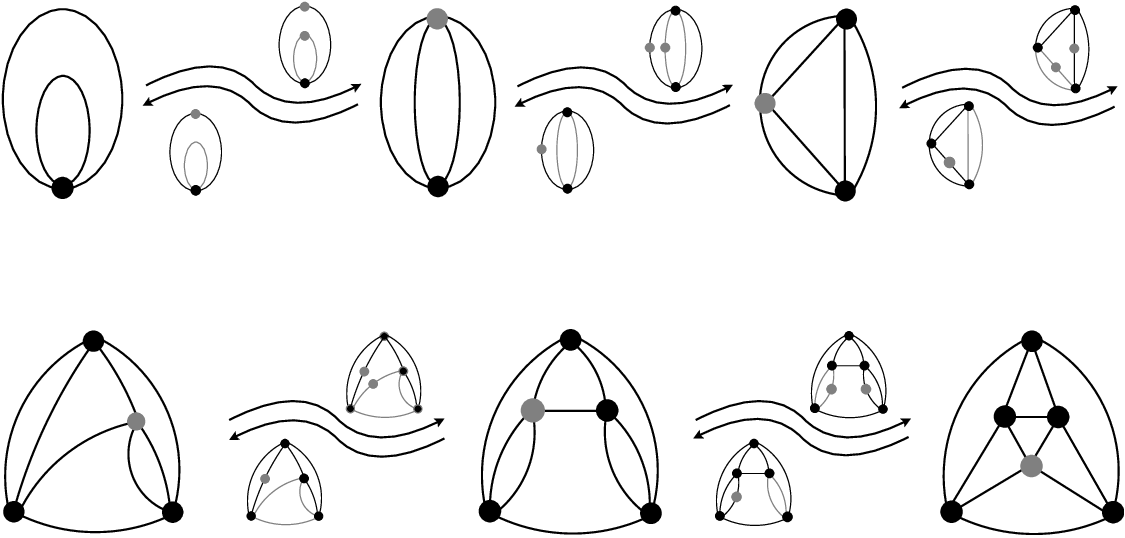}
\caption{A recursive characterization of
the Octahedron from the bouquet $B_2$.}
\end{figure}
Since simple graphs are a proper subset of graphs, simple 4-regular planar graphs are also recursively characterized from the Octahedron. However, we continuously take the operation defined in Corollary~\ref{MC-5} for getting from one simple 4-regular planar graph to another, see the following for an example.
\begin{myexample}\label{ME-7}
Reading Figure~7 from left to right is a constructing the simple 4-regular plane on lower right from the Octahedron, and looking in reverse order is reducing the plane graph to the Octahedron. The function of the small pictures and gray color are identical to Example~\ref{ME-6}.
\end{myexample}

\begin{figure}[htbp]
\centering
\includegraphics[width=16.0cm]{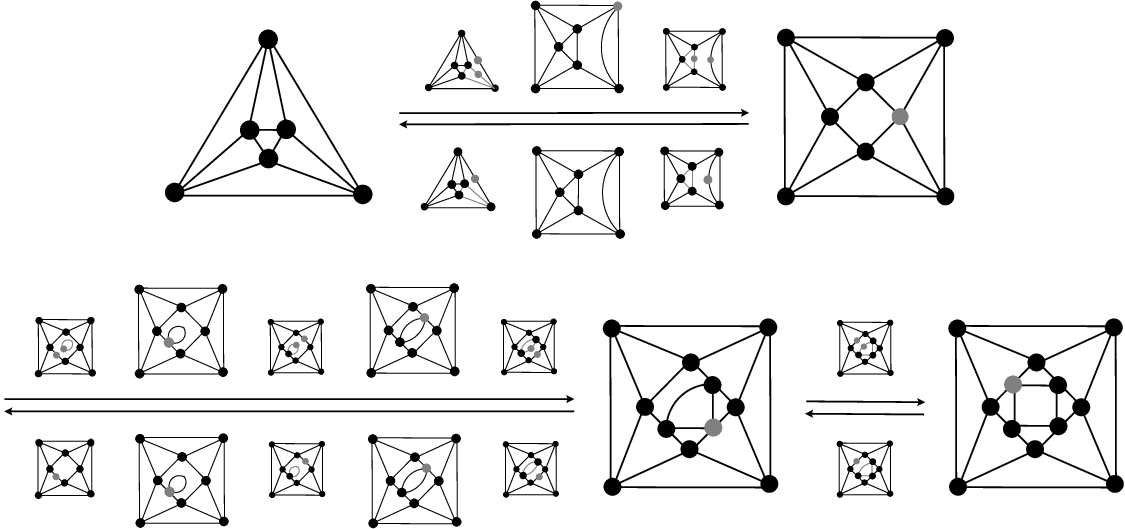}
\caption{An example of recursive characterization of simple 4-regular planar graphs.}
\end{figure}

\vskip0.2cm

\noindent{\bf Acknowledgments} This work is supported by National Natural Science Foundation of China,  Grant/Award Numbers: 11961070, 12101600.


\newpage\begin{thebibliography}{10}
\bibitem{Bondy} J. Bondy and U. Murty, {\it Graph theory}, Springer, Berlin, 2008.
\bibitem{Broersma} H. J. Broersma, A. J. W. Duijvestijn, and F. G\"{o}bel, {\it Generating all 3-connected 4-regular planar graphs from the octahedron graph}, J. Graph Theory {\bf 17} (1993), 613-620.
\bibitem{Chartrand} G. Chartrand and F. Harary, {\it Planar permutation graphs}, Ann. Inst. Henri Poincar\'{e} Sec B {\bf 3} (1967), 433-438.
\bibitem{Chudnovsky} M. Chudnovsky, G. Kalai, E. Nevo, I. Novik, and P. Seymour, {\it Bipartite minors}, J. Comb. Theory Ser. B {\bf 116} (2016), 219-228.
\bibitem{Lehel} J. Lehel, {\it Generating all 4-regular planar graphs from the graph of the Octahedron}, J. Graph Theory {\bf 5} (1981), 423-426.
\bibitem{Manca} P. Manca, {\it Generating all planar graphs regular of degree four}, J. Graph Theory {\bf 3} (1979), 357-364.
\bibitem{MM} M. Metsidik and X. Jin, {\it Eulerian and even-face ribbon graph minors},  Discrete Math {\bf 343} (2020) 111953.
\bibitem{Moffatt6} I. Moffatt, {\it Excluded minors and the ribbon graphs of knots}, J. Graph Theory {\bf 81} (2016), 329-341.
\bibitem{Moffatt7} I. Moffatt, {\it Ribbon graph minors and low-genus partial duals} Ann. Comb. {\bf 20} (2016), 373-378.
\bibitem{R-S} N. Robertson and P. Seymour, {\it Graph minors. XX. Wagner's Conjecture}, J. Combin. Theory Ser. B {\bf 92} (2004), 325-357.
\bibitem{Thomassen} C. Thomassen, {\it Kuratowski's theorem}, J. Graph Theory {\bf 5} (1981), 225-241.
\bibitem{O.Ve} O. Veblem, {\it An application of modular equations in analysis situs}, Ann. Math. {\bf 14} (1912), 86-94.
\bibitem{Wagner} K. Wagner, {\it\"{U}ber eine Eigenschaft der ebenen Komplexe}, Math. Ann. {\bf 114} (1) (1937), 570-590.
\bibitem{Wagner1} D.K. Wagner, {\it Bipartite and Eulerian minors}, Eur. J. Combin. {\bf 74} (2018), 1-10.
\end{thebibliography}
\end{document}